\documentclass[11pt]{amsart}
\usepackage{amssymb,latexsym}
\usepackage{url,units}
\usepackage[margin=1in]{geometry}
\newcommand{\EME}{Erd\H{o}s-Moser equation}

\DeclareMathOperator{\LCM}{LCM}
\newcommand{\emdash}{\hspace{2pt}\textemdash\hspace{2pt}}

\newcommand{\osn}[1]{\oldstylenums{#1}}
\theoremstyle{plain}
\newtheorem{thm}{Theorem}
\newtheorem{cor}{Corollary}
\newtheorem{lem}{Lemma}
\newtheorem{prop}{Proposition}
\newtheorem{cnj}{Conjecture}
\theoremstyle{definition}
\newtheorem{dfn}{Definition}
\newtheorem{remark}{Remark}
\newtheorem{example}{Example}
\numberwithin{equation}{section}

\title[Reducing The Equation $1^n + 2^n + \dotsb + k^n = (k+1)^n$ Modulo $k$ and $k^2$]{Reducing the Erd\H{o}s-Moser Equation\\ $\boldsymbol{1^n + 2^n + \dotsb + k^n = (k+1)^n}$\\ Modulo $\boldsymbol{k}$ and $\boldsymbol{k^2}$}
\author{Jonathan Sondow}
\address{209 West 97th Street\\ New York\\ NY 10025\\ USA}
\email{jsondow@alumni.princeton.edu}
\author{Kieren MacMillan}
\address{49 Lessard Avenue\\ Toronto\\ Ontario\\ Canada\\ M6S 1X6}
\email{kieren@alumni.rice.edu}

\begin{document}
\baselineskip=1.33em
\begin{abstract}
An open conjecture of Erd\H{o}s and Moser is that the only solution of the Diophantine equation in the title is the trivial solution $1 + 2 = 3$. Reducing the equation modulo $k$ and $k^2$, we give necessary and sufficient conditions on solutions to the resulting congruence and supercongruence. A corollary is a new proof of Moser's result that the conjecture is true for odd exponents $n$. We also connect solutions $k$ of the congruence to primary pseudoperfect numbers and to a result of Zagier. The proofs use divisibility properties of power sums as well as Lerch's relation between Fermat and Wilson quotients.
\end{abstract}
\maketitle

%%%  SEC: INTRODUCTION
\section{Introduction}  \label{SEC: Intro}
Around \osn{1953}, Erd\H{o}s and Moser studied the Diophantine equation
\begin{equation}
    1^n + 2^n + \dotsb + k^n = (k+1)^n  \label{EQ: EME}
\end{equation}
and made the following conjecture.

\begin{cnj} \label{CNJ: EMC}
The only solution of \eqref{EQ: EME} in positive integers is the trivial solution \mbox{$1 + 2 = 3$}.
\end{cnj}

Moser~\cite{Moser} proved the statement when $n$ is odd or $k<10^{10^6}\!$. In \osn{1987} Schinzel showed that in any solution, $k$ is even~\cite[p.~800]{MoreeEtAl}. An extension of Schinzel's theorem to a generalization of equation \eqref{EQ: EME} was given in \osn{1996} by Moree \cite[Proposition~9]{Moree}. For a recent elementary proof of a special case, see MacMillan and Sondow \cite[Proof of Proposition~2]{MSgeme}.

Many other results on the \emph{\EME{}}~\eqref{EQ: EME} are known, but it has not even been established that there are only finitely many solutions. For surveys of work on this and related problems, see Butske, Jaje, and Mayernik~\cite{Butske}, Guy~\cite[D7]{Guy}, and Moree~\cite{MoreeOnMoser}.

In the present paper, we first approximate equation~\eqref{EQ: EME} by the congruence
\begin{equation}
    1^n+ 2^n+ \dotsb + k^n\equiv (k+1)^n \pmod{k}.  \label{EQ: EMECong}
\end{equation}
In Section~\ref{SEC: Congruence}, we give necessary and sufficient conditions on $n$ and $k$ (Theorem~\ref{THM: EMCongT}), and we show that if a solution $k$ factors into a product of $1,2,3$, or $4$ primes, then $k=2,6,42$, or $1806$, respectively (Proposition~\ref{PROP: r=1,2,3,4 gives k=2,6,42,1806}). In Section~\ref{SEC: PPP}, we relate solutions $k$ to primary pseudoperfect numbers and to a result of Zagier. In the final section, Theorem~\ref{THM: EMCongT} is extended to the supercongruence (Theorem~\ref{THM: p^2})
\begin{equation*}
    1^n+ 2^n+ \dotsb + k^n\equiv (k+1)^n \pmod{k^2}.  \label{EQ: EMESuperCong}
\end{equation*}

Our methods involve divisibility properties of power sums, as well as Lerch's formula relating Fermat and Wilson quotients.

As applications of Theorems~\ref{THM: EMCongT} and~\ref{THM: p^2}, we reprove Moser's result that Conjecture~\ref{CNJ: EMC} is true for odd exponents~$n$ (Corollary~\ref{COR: EMC n odd}), and for even $n$ we show that a solution $k$ to \eqref{EQ: EME} cannot be a primary pseudoperfect number with $8$ or fewer prime factors (Corollary~\ref{COR: p^2 applic}). In a paper in preparation, we will give other applications of our results to the \EME{}.

%%  SEC: CONGRUENCE
\section{The Congruence  $1^n + 2^n + \dotsb + k^n \equiv (k+1)^n\!\pmod{k}$} \label{SEC: Congruence}

We will need a classical lemma on power sums. (An empty sum will represent $0$, as usual.)

\begin{dfn}
For integers $n\ge0$ and $a\ge1$, let $\Sigma_n(a)$ denote the \emph{power sum}
\begin{equation*}
	\Sigma_n(a) := 1^n + 2^n + \dotsb + a^n.
\end{equation*}
Set $\Sigma_n(0):=0$.
\end{dfn}

\begin{lem}  \label{LEM: H&W Thm119}
If $n$ is a positive integer and $p$ is a prime, then
\begin{align*}
	\Sigma_n(p) \equiv
	\begin{cases}
		\, -1 \pmod{p}, \qquad (p-1) \mid n, \\
		\, \hspace{0.75em} 0 \pmod{p}, \qquad (p-1) \nmid n.
	\end{cases}
\end{align*}
\end{lem}
\begin{proof}
Hardy and Wright's proof \cite[Theorem 119]{HW} uses primitive roots. For a very elementary proof, see MacMillan and Sondow \cite{MSpascal}.
\end{proof}

We now give necessary and sufficient conditions on solutions to~\eqref{EQ: EMECong}.

\begin{thm} \label{THM: EMCongT}
Given positive integers $n$ and $k$, the congruence
\begin{equation}
	1^n + 2^n + \dotsb + k^n \equiv (k+1)^n \pmod{k}  \label{EQ: EMCong}
\end{equation}
holds if and only if prime $p\mid k$ implies
\begin{enumerate}
	\item[(i).] $n \equiv 0 \pmod{(p-1)}$, and  \label{EMCongT, n}
	\item[(ii).] $\dfrac{k}{p}+1 \equiv 0 \pmod{p}$.  \label{EMCongT, k/p}
\end{enumerate}
In that case $k$ is square-free, and if $n$ is odd, then $k=1$ or $2$.
\end{thm}

\begin{proof}
First note that if $n,k,p$ are \emph{any} positive integers with $p\mid k$, then
\begin{equation}
	\Sigma_n(k) = \sum_{h=1}^{k/p} \sum_{j=1}^p ((h-1)p+j)^n \equiv \frac{k}{p}\, \Sigma_n(p) \pmod{p}.  \label{EQ: factor out k/p}
\end{equation}

Now assume that (i) and (ii) hold whenever prime $p\mid k$. Then, using Lemma~\ref{LEM: H&W Thm119}, both $\Sigma_n(p)$ and $k/p$ are congruent to $-1$ modulo $p$, and so $\Sigma_n(k) \equiv 1\!\pmod{p}$, by \eqref{EQ: factor out k/p}. Thus, as (ii) implies $k$~is square-free, $k$ is a product of distinct primes each of which divides $\Sigma_n(k)-1$. It follows that $\Sigma_n(k) \equiv 1\!\pmod{k}$, implying \eqref{EQ: EMCong}.

Conversely, assume that \eqref{EQ: EMCong} holds, so that $\Sigma_n(k) \equiv 1\!\pmod{k}$. If prime $p\mid k$, then \eqref{EQ: factor out k/p} gives $\dfrac kp\, \Sigma_n(p) \equiv 1\!\pmod{p}$, and so $\Sigma_n(p)\not\equiv 0\!\pmod{p}$. Now Lemma~\ref{LEM: H&W Thm119} yields both $(p-1) \mid n$, proving~(i), and $\Sigma_n(p) \equiv -1\!\pmod{p}$, implying (ii).

If $n$ is odd, then by (i) no odd prime divides $k$. As $k$ is square-free, $k=1$ or $2$.
\end{proof}

\begin{cor} \label{COR: EMC n odd}
The only solution of the \EME{} with odd exponent~$n$ is $1+2=3$.
\end{cor}
\begin{proof}
Given a solution with $n$ odd, Theorem~\ref{THM: EMCongT} implies $k=1$ or $2$. But $k=1$ is clearly impossible, and $k=2$ evidently forces $n=1$.
\end{proof}

Recall that, when $x$ and $y$ are real numbers, $x\equiv y\!\pmod{1}$ means that $x-y$ is an integer.

\begin{cor} \label{COR: which k and n}
A given positive integer $k$ satisfies the congruence \eqref{EQ: EMCong}, for some exponent $n$, if and only if the Egyptian fraction congruence
\begin{equation}
	\frac{1}{k} + \sum_{p\mid k} \frac{1}{p} \equiv 1 \pmod{1}  \label{EQ: frac mod 1}
\end{equation}
holds, where the summation is over all primes $p$ dividing $k$. In that case, $k$ is square-free, and $n$ is any number divisible by the least common multiple $\LCM\{p-1:\,\text{prime}\,\, p\mid k\}$.
\end{cor}
\begin{proof}
Condition~\eqref{EQ: frac mod 1} is equivalent to the congruence
\begin{equation}
	1+\sum_{p\mid k} \frac{k}{p} \equiv 0 \pmod{k},  \label{EQ: frac mod k}
\end{equation}
which in turn is equivalent to condition (ii) in Theorem~\ref{THM: EMCongT}, since each implies $k$ is square-free. The theorem now implies the corollary.
\end{proof}

\begin{remark}
In \eqref{EQ: frac mod 1} we write $\equiv 1\!\pmod{1}$, rather than the equivalent $\equiv 0\!\pmod{1}$, in order to contrast the condition with that in Definition~\ref{DFN: ppp} of the next section.
\end{remark}

For solutions to the congruence~\eqref{EQ: EMCong}, we determine the possible values of $k$ with at most four (distinct) prime factors. First we prove a lemma. (An empty product will represent $1$, as usual.)

%%  LEM: product-sum 
\begin{lem}  \label{LEM: product-sum}
Let $k=p_1p_2\dotsb p_r$, where the $p_i$ are primes. If $k$ satisfies the integrality condition~\eqref{EQ: frac mod 1}, then for any subset $S\subset \{p_1,\dotsc,p_r\}$, there exists an integer $q=q(r,S)$ such that
\begin{equation}
	q\prod_{p\in S} p = 1 + \sum_{p\in S} \frac{k}{p}.  \label{EQ: CRT}
\end{equation}
\end{lem}
\begin{proof}
This follows from Corollary~\ref{COR: which k and n} and Theorem~\ref{THM: EMCongT} condition ~(ii), using the Chinese Remainder Theorem. For an alternate proof, denote the summation in \eqref{EQ: CRT} by $\Sigma$, and note that \eqref{EQ: frac mod 1} implies $1+\Sigma\equiv 1+\dfrac kp\equiv 0\!\pmod{p}$, for each $p\in S$. Then, since $k$ is square-free, $\prod_{p\in S}p$ divides $1+\Sigma$, and the lemma follows.
\end{proof}

%%  PROP: r=1,2,3,4 gives k=2,6,42,1806
\begin{prop}  \label{PROP: r=1,2,3,4 gives k=2,6,42,1806}
Let $k$ be a product of $r$ primes. Suppose $1^n + 2^n + \dotsb + k^n \equiv (k+1)^n\!\pmod{k}$, for some exponent~$n>0$; equivalently, suppose \eqref{EQ: frac mod 1} holds. If $r=0,1,2,3,4$, then $k=1,2,6,42,1806$, respectively.
\end{prop}
\begin{proof}
Theorem~\ref{THM: EMCongT} implies $k = p_1p_2 \dotsm p_r$, where $p_1 < p_2 < \dotsb < p_r$ are primes.

($r=0,1$). If $r=0$, then $k=1$. If $r=1$, then $k=p_1$ is prime, and \eqref{EQ: frac mod k} yields $k \mid 2$, so that \mbox{$k=2$}.

($r=2$). For $k=p_1p_2$, congruence~\eqref{EQ: frac mod k} gives $p_2 \mid (p_1+1)$. Since $p_2 \ge p_1+1$, it follows that $p_2 = p_1+1$. As $p_2$ and $p_1$ are prime, $p_2=3$ and $p_1=2$, and hence $k=6$.

($r=3,4$). In general, if $k = p_1\dotsm p_r$, where $p_1< \dotsm <p_r$ are primes, then by Lemma~\ref{LEM: product-sum}, for $i=1,\dotsc,r$ there exists an integer $q_i$ such that $q_ip_i=\dfrac{k}{p_i}+1$. In particular,
\begin{equation}
	q_rp_r=p_1\dotsm p_{r-2}p_{r-1}+1.  \label{EQ: q_rp_r}
\end{equation}
Hence if $r>2$, so that $p_{r-1}<p_r-1$, then $q_r<p_1\dotsm p_{r-2}$. We also have
\begin{equation*}
	q_rq_{r-1}p_{r-1}=q_r(p_1\dotsm p_{r-2}p_r+1)=p_1\dotsm p_{r-2}(p_1\dotsm p_{r-1}+1)+q_r,
\end{equation*}
and so
\begin{equation}
	p_{r-2}< p_{r-1}\mid(p_1\dotsm p_{r-2}+q_r)<2p_1\dotsm p_{r-2}.  \label{EQ: mid<}
\end{equation}

Now take $r=3$, so that $k=p_1p_2p_3$. Then $q_2p_2=p_1p_3+1 $ and $q_3p_3=p_1p_2+1$, for some integers $q_2$ and $q_3$. By  \eqref{EQ: mid<} we have $p_1<p_2\mid(p_1+q_3)<2p_1$. Hence $p_2=p_1+q_3$. Substituting $q_3=p_2-p_1$ into $q_3p_3=p_1p_2+1$ yields $p_1\mid(p_2p_3-1)$. As $p_1\mid(p_2p_3+1)$, we conclude that $p_1\mid2$. Therefore $p_1=2$. Then $p_3 \mid (2p_2+1)$. As $p_3>p_2$, we get $p_3=2p_2+1$. Then we have $p_2\mid (2p_3+1)=4p_2+3$, and so $p_2\mid 3$. Therefore $p_2=3$, and hence $p_3=7$. Thus $k=2\cdot3\cdot7=42$.

Finally, take $r=4$. Lemma~\ref{LEM: product-sum} with $S=\{p_1,\dotsc,p_r\}$ and $r=4$ gives an integer~$q$ such that
\begin{equation*}
	qp_1p_2p_3p_4=p_1(p_2p_3+p_2p_4+p_3p_4)+p_2p_3p_4+1<4p_2p_3p_4,
\end{equation*}
so that $qp_1<4$. Hence $q=1$, and $p_1=2$ or $3$. Now
\begin{equation*}
(p_1-1)p_2p_3p_4=p_1(p_2p_3+p_2p_4+p_3p_4)+1<3p_1p_3p_4,
\end{equation*}
and so $(p_1-1)p_2<3p_1$. This implies $p_1=2$ and $p_2=3$. Thus $k=2\cdot3p_3p_4$. Then \eqref{EQ: q_rp_r} and \eqref{EQ: mid<} give $q_4p_4=6p_3+1$ and $3<p_3\mid(6+q_4)<12$. The only solution is $(q_4,p_3,p_4)=(1,7,43)$, and so $k=2\cdot3\cdot7\cdot43=1806$. This completes the proof.
\end{proof}

\begin{example} \label{EX: r=5} Take $k=1806$ in Corollary~\ref{COR: which k and n}. Since $1806 = 2\cdot3\cdot7\cdot43$ and
\begin{equation*}
	\frac{1}{1806} + \frac{1}{2} + \frac{1}{3} + \frac{1}{7} + \frac{1}{43} = 1,
\end{equation*}
and since $\LCM(1,2,6,42)=42$, one solution of ~\eqref{EQ: EMCong} is
\begin{equation*}
	1^{42} + 2^{42} + \dotsb + 1806^{42} \equiv 1807^{42} \pmod{1806}.
\end{equation*}
\end{example}

%%%  SEC: PPP NUMBERS
\section{Primary Pseudoperfect Numbers} \label{SEC: PPP}

Recall that a positive integer is called \emph{perfect} if it is the sum of \emph{all} its proper divisors, and \emph{pseudoperfect} if it is the sum of \emph{some} of its proper divisors~\cite[B2]{Guy}.

\begin{dfn}  \label{DFN: ppp}
(From \cite{Butske}.) A \emph{primary pseudoperfect number} is an integer $K>1$ that satisfies the Egyptian fraction equation
\begin{equation}
	\frac{1}{K} + \sum_{p \mid K} \frac{1}{p} = 1,  \label{EQ: ppp def}
\end{equation}
where the summation is over all primes dividing $K$.
\end{dfn}

Multiplying \eqref{EQ: ppp def} by $K$, we see that $K$ is square-free, and that every primary pseudoperfect number, except $2$, is pseudoperfect.

\begin{cor} \label{COR: PPPs are solutions}
Every primary pseudoperfect number $K$ is a solution to the congruence~\eqref{EQ: EMCong}, for some exponent~$n$.
\end{cor}

\begin{proof}
This is immediate from Definition~\ref{DFN: ppp} and Corollary~\ref{COR: which k and n}.
\end{proof}

A priori, the equality~\eqref{EQ: ppp def} is a stronger condition than the congruence~\eqref{EQ: frac mod 1} in Corollary~\ref{COR: which k and n}. However, \eqref{EQ: ppp def} and~\eqref{EQ: frac mod 1} may in fact be equivalent, because all the known solutions of \eqref{EQ: frac mod 1} also satisfy~\eqref{EQ: ppp def}\emdash see~\cite{Butske}. In other words, primary pseudoperfect numbers may be the only solutions~$k$ to the congruence~\eqref{EQ: EMCong}, except for $k=1$.

According to~\cite{Butske}, Table~1 contains all primary pseudoperfect numbers $K$ with $r \le 8$ (distinct) prime factors. In particular, for each $r=1,2,\ldots,8$, there exists exactly one such $K$ (as conjectured by Ke and Sun~\cite{Ke_Sun} and by Cao, Liu, and Zhang~\cite{Cao_Liu_Zhang}). No $K$ with $r>8$ prime factors is known. As with perfect numbers, no odd primary pseudoperfect number has been discovered.

\begin{center}
\begin{table}[ht]
\caption{(from~\cite{Butske}) The primary pseudoperfect numbers $K$ with $r \le 8$ prime factors}
	\label{TABLE: PPPs}
\begin{tabular}{|c|c|c|}
	\hline
	\textit{r} & \textit{K} & prime factors  \\
	\hline
	$1$ & $2$ & $2$  \\
	$2$ & $6$ & $2,3$  \\
	$3$ & $42$ & $2,3,7$  \\
	$4$ & $1806$ & $2,3,7,43$  \\
	$5$ & $47058$ & $2,3,11,23,31$  \\
	$6$ & $2214502422$ & $2,3,11,23,31,47059$  \\
	$7$ & $52495396602$ & $2,3,11,17,101,149,3109$  \\
	$8$ & $8490421583559688410706771261086$ & $2,3,11,23,31,47059,2217342227,1729101023519$  \\
	\hline
\end{tabular}
\end{table}
\end{center}
\vspace{-1.7em}

Table~\ref{TABLE: PPPs} was obtained in~\cite{Butske} using computation and computer search techniques. Note that the cases $r=1,2,3,4$ follow a fortiori from our Proposition~\ref{PROP: r=1,2,3,4 gives k=2,6,42,1806}.

\begin{center}
\begin{table}[ht]
\caption{The known solutions to $1^n + 2^n + \dotsb + k^n \equiv (k+1)^n\!\pmod{k}$}
	\label{TABLE: EMCong solutions}
\begin{tabular}{|c|c|}
	\hline
	\textit{k} & \textit{n} is any multiple of  \\
	\hline
	$1$ & $1$  \\
	$2$ & $1$  \\
	$6$ & $2$  \\
	$42$ & $6$  \\
	$1806$ & $42$  \\
	$47058$ & $330$  \\
	$2214502422$ & $235290$  \\
	$52495396602$ & $310800$  \\
	$8490421583559688410706771261086$ & $1863851053628494074457830$  \\
	\hline
\end{tabular}
\end{table}
\end{center}
\vspace{-1.7em}

Table~\ref{TABLE: EMCong solutions} was calculated from Table~\ref{TABLE: PPPs}, using Corollaries~\ref{COR: which k and n} and~\ref{COR: PPPs are solutions}.

\begin{example} \label{EX: r=8} The simplest case of the congruence~\eqref{EQ: EMCong} in which $k$ has $8$ prime factors is
\begin{align*}
	\sum_{j=1}^{8490421583559688410706771261086} j^{1863851053628494074457830}& \\
		\equiv 8490421583559688410706771261087&^{1863851053628494074457830} \\
		\pmod{&8490421583559688410706771261086}.
\end{align*}
\end{example}

Zagier gave three characterizations of the numbers $1,2,6,42,1806$. We add two more.

\begin{prop} \label{PROP: Zagier}
Each of the following five conditions is equivalent to $k\in\{1,2,6,42,1806\}$.
\begin{enumerate}
	\item The congruence $a^{k+1} \equiv a\!\pmod{k}$ holds, for all $a$.  \label{Zagier, j^{k+1}}
	\item $k=p_1p_2\dotsm p_r$, where $r\ge 0$, the $p_i$ are distinct primes, and \mbox{$(p_i-1)\mid k$}.  \label{Zagier, squarefree}
 	\item $k=p_1p_2\dotsm p_r$, where $r\ge 0$ and \mbox{$p_i=p_1 \dotsm p_{i-1}+1$} is prime, \mbox{$i = 1,\dotsc,r$}.  \label{Zagier, product+1}
	\item $k$ is a product of at most $4$ primes, and $1^n + 2^n + \dotsb + k^n \equiv (k+1)^n\!\pmod{k}$, for some exponent~$n$.  \label{Zagier, EMCongr}
	\item Either $k=1$ or $k$ is a primary pseudoperfect number with $4$ or fewer prime factors.  \label{Zagier, PPP}
\end{enumerate}
\end{prop}
\begin{proof}
For (i), (ii), (iii), see the solution to the first problem of Zagier~\cite{Zagier}. Proposition \ref{PROP: r=1,2,3,4 gives k=2,6,42,1806} yields~(iv). Corollary~\ref{COR: PPPs are solutions} and (iv) give~(v).
\end{proof}

%%%  SEC: SUPERCONGRUENCES
\section{Supercongruences}  \label{SEC: Super}

If the conditions in Theorem \ref{THM: EMCongT} are satisfied, the following corollary shows that the congruence~\eqref{EQ: factor out k/p} can be replaced with a ``supercongruence'' (compare Zudilin~\cite{WZsuper}).

\begin{cor} \label{COR: Sigma_n(k) mod p^2}
If $1^n + 2^n + \dotsb + k^n\equiv (k+1)^n\!\pmod{k}$ and prime $p \mid k$, then
\begin{equation}
 	\Sigma_n(k) \equiv \frac{k}{p}\, \Sigma_n(p) \pmod{p^2}.  \label{EQ: p^2}
\end{equation}
\end{cor}
\begin{proof}
By Theorem \ref{THM: EMCongT}, it suffices to prove the more general statement that, if prime $p\mid k$ and $(p-1)\mid n$, and if either $k=2$ or $n$ is even, then~\eqref{EQ: p^2} holds. Set $a=k/p$ in the equation~\eqref{EQ: factor out k/p}. Expanding and summing, we see that
\begin{equation*}
 	\Sigma_n(k)\equiv a\Sigma_n(p)+\frac{1}{2}\,a(a-1)np\Sigma_{n-1}(p) \pmod{p^2}.
\end{equation*}
If $p>2$, then $(p-1)\mid n$ implies $(p-1)\nmid (n-1)$, and Lemma~\ref{LEM: H&W Thm119} gives $p\mid \Sigma_{n-1}(p)$. In case $p=2$, either $a=k/2=1$ or $2\mid n$, and each implies $2\mid (1/2)a(a-1)n$. In all cases, \eqref{EQ: p^2} follows.
\end{proof}

For an extension of Theorem \ref{THM: EMCongT} itself to a supercongruence, we need a definition and a lemma.

\begin{dfn}
By Fermat's and Wilson's theorems, for any prime $p$ the \emph{Fermat quotient}
\begin{align}
	q_p(j) &:= \frac{j^{p-1}-1}{p}, \quad p\nmid j,  \label{EQ: qDef}
\intertext{and the \emph{Wilson quotient}}
 	W_p &:= \frac{(p-1)!+1}{p}  \notag
\end{align}
are integers.
\end{dfn}

\begin{lem}[Lerch~\cite{Lerch}] \label{LEM: Lerch}
If $p$ is an odd prime, then the Fermat and Wilson quotients are related by \emph{Lerch's formula}
\begin{equation*}
	\sum_{j=1}^{p-1} q_p(j) \equiv W_p \pmod{p}.
\end{equation*}
\end{lem}
\begin{proof}
Given $a$ and $b$ with $p\nmid ab$, set $j=ab$ in \eqref{EQ: qDef}. Substituting $a^{p-1}=pq_p(a)+1$ and $b^{p-1}=pq_p(b)+1$, we deduce Eisenstein's relation \cite{Eisenstein}
\begin{align*}
	q_p(ab) &\equiv q_p(a) + q_p(b) \pmod{p},
\intertext{ which implies} 
	q_p((p-1)!) &\equiv \sum_{j=1}^{p-1} q_p(j) \pmod{p}.
\end{align*}
On the other hand, setting $j=(p-1)!=pW_p-1$ in \eqref{EQ: qDef} and expanding, the hypothesis $p-1\ge2$ leads to $q_p((p-1)!) \equiv W_p\!\pmod{p}$. This proves the lemma.
\end{proof}

\begin{thm} \label{THM: p^2}
For $n=1$, the supercongruence
\begin{equation}
	1^n + 2^n + \dotsb + k^n\equiv (k+1)^n \pmod{k^2}  \label{EQ: EMCong k^2}
\end{equation}
holds if and only if $k=1$ or $2$. For $n\ge3$ odd,~\eqref{EQ: EMCong k^2} holds if and only if $k=1$. Finally, for $n\ge2$ even,~\eqref{EQ: EMCong k^2} holds if and only if prime $p\mid k$ implies
\begin{enumerate}
	\item[(i).] $n \equiv 0\!\pmod{(p-1)}$, and  \\[-0.75em]
	\item[(ii).] $\dfrac{k}{p}+1 \equiv p\left(n(W_p+1)-1\right)\pmod{p^2}.$
\end{enumerate}
\end{thm}
\begin{proof}
To prove the first two statements, use Theorem \ref{THM: EMCongT} together with the fact that the congruences $1^n+2^n\equiv 1\!\pmod{4}$ and $3^n\equiv (-1)^n \equiv -1\!\pmod{4}$ all hold when $n\ge3$ is odd.

Now assume $n\ge2$ is even. Let $p$ denote a prime. By Theorem \ref{THM: EMCongT}, we may assume that (i) holds if $p\mid k$, and that $k$ is square-free. It follows that the supercongruence \eqref{EQ: EMCong k^2} is equivalent to the system
\begin{align*}
	\Sigma_n(k) &\equiv (k+1)^n \pmod{p^2}, \qquad p\mid k.
\intertext{Corollary~\ref{COR: Sigma_n(k) mod p^2} and expansion of $(k+1)^n$ allow us to write the system as}
	\frac{k}{p}\, \Sigma_n(p) &\equiv 1+nk \pmod{p^2}, \qquad p\mid k.
\intertext{\indent Since $n$ is at least $2$ and $(p-1)\mid n$, we have}
	\Sigma_n(p) &\equiv \Sigma_n(p-1) \pmod{p^2}  \\
		&= \sum_{j=1}^{p-1} (j^{p-1})^{n/(p-1)}.
\end{align*}
Substituting $j^{p-1} = 1+pq_p(j)$ and expanding, the result is
\begin{align}
	\Sigma_n(p)
		\equiv \sum_{j=1}^{p-1} \biggl(1+\frac{n}{p-1}pq_p(j)\biggr)
		&\equiv p-1-np\sum_{j=1}^{p-1} q_p(j)  \pmod{p^2},
\end{align}
since $n/(p-1) \equiv -n\!\pmod{p}$. Now Lerch's formula (if $p$ is odd), together with the equality $q_2(1)=0$ and the evenness of $n$ (if $p=2$), yield
\begin{align}
	\Sigma_n(p) &\equiv p-1-npW_p \pmod{p^2}.  \notag
\intertext{\indent Summarizing, the supercongruence \eqref{EQ: EMCong k^2} is equivalent to the system}
	\frac kp\, (p-1-npW_p) &\equiv 1+nk \pmod{p^2}, \qquad p\mid k.  \notag
\intertext{It in turn can be written as}
	\frac kp+1 &\equiv -k\bigl(n(W_p+1)-1\bigr) \pmod{p^2}, \qquad p\mid k. \label{EQ: k==-p}
\end{align}
On the right-hand side, we substitute $k \equiv -p\!\pmod{p^2}$ (deduced from \eqref{EQ: k==-p} multiplied by $p$), and arrive at~(ii). This completes the proof.
\end{proof}

\newpage
\begin{cor}  \label{COR: p^2 applic}
Let $n\ge 2$ be even and let $K$ be a primary pseudoperfect number with $r\le8$ prime factors.
\begin{enumerate}
 	\item[(i).] Then $(n,K)$ is a solution of the supercongruence~\eqref{EQ: EMCong k^2} if and only if either $K=2$, or $K=42$ and $n \equiv 12\!\pmod{42}$.
	\item[(ii).] The supercongruence
		\begin{equation}
 			1^n + 2^n + \dotsb + K^n\equiv (K+1)^n \pmod{K^3}  \label{EQ: EMCong k^3}
		\end{equation}
	holds if and only if $K=2$ and $n\ge 4$.
	\item[(iii).] The \EME{} has no solution $(n,k)$ with $k=K$.
\end{enumerate}
\end{cor}
\begin{proof}
(i). We use Table $1$.

($r=1$). Theorem~\ref{THM: p^2} with $k=p=2$ implies $(n,2)$ is a solution to~\eqref{EQ: EMCong k^2}. (This can also be seen directly from \eqref{EQ: EMCong k^2}: both sides are congruent to $1$ modulo~$4$.)

($r=2$). Suppose $k=2\cdot3$ is a solution to~\eqref{EQ: EMCong k^2}. Since $2\mid n$, condition~(ii) in Theorem~\ref{THM: p^2} with $p=2$ gives $3+1=\dfrac kp+1 \equiv -2\!\pmod{4}$, a contradiction. Therefore, there is no solution with $k=6$.

($r=3$). For $k=2\cdot3\cdot7$, condition (i) in Theorem~\ref{THM: p^2} requires $6\mid n$. Then~(ii) is satisfied for $p=2$ and $3$. For $p=7$, we need $6+1\equiv 7\left((103+1)n-1\right)\!\pmod{49}$, which reduces to $3n\equiv 1\!\pmod{7}$. Since also $6\mid n$, only $n\equiv 12\!\pmod{42}$ gives a solution with $k=42$.

($r=4$). If $k=2\cdot3\cdot7\cdot43$, condition (ii) with $p=2$ rules out any solution.

($r=5$). For $k=2\cdot3\cdot11\cdot23\cdot31$, condition (i) gives $3\mid n$. As $\dfrac k3+1\equiv 0\not\equiv -3\!\pmod{9}$, by~(ii) there is no solution.

($r=6,7,8$). For the numbers $K$ in Table 1 with $r=6,7,8$ prime factors, conditions~(i) and~(ii) require $\dfrac kp+1 \equiv -p\!\pmod{p^2}$, for $p=2,3,2$, respectively. But the requirement is violated in each case, and so no solution exists.

(ii). Part (i) implies that the only \emph{possible} solutions $(n,K)$ of \eqref{EQ: EMCong k^3} are $K=2$, and $K=42$ with $n \equiv 12\!\pmod{42}$.

It is easy to check that $(n,K)=(2,2)$ is not a solution. To see that $(n,2)$ is a solution when $n\ge 4$ is even, we need to show that $1+2^n \equiv 3^n\!\pmod{2^3}$. Since $1+2^n = 1+4^{n/2} \equiv 1\!\pmod{8}$ and $3^n = 9^{\,n/2} = (1+8)^{n/2} \equiv 1\!\pmod{8}$, the case $K=2$ is proved.

Now suppose \eqref{EQ: EMCong k^3} holds with $K=42$ and $n \equiv 12\!\pmod{42}$. Since
\begin{align}
	(K+1)^n &\equiv 1+nK+\frac{1}{2}n(n-1)K^2 \pmod{K^3}, \notag
\intertext{by setting $n=6n_1$ we infer that}
	\Sigma_{n}(42) &\equiv 1 - 5040n_1 + 31752n_1^2 \equiv 1 \pmod{2^3}. \label{EQ: K=42}
\end{align}
But as $n\ge 4$ is even, each of the $42$ terms in the sum $\Sigma_n(42)$ is congruent to $1$ or $0$ modulo~$8$ according as the term is odd or even, and so $\Sigma_n(42) \equiv 21\!\pmod{8}$. This contradicts \eqref{EQ: K=42}, proving~(ii).

(iii). This follows from (ii) and the fact that if $k=2$ in the \EME{}, then evidently $n=1$.
\end{proof}

\newpage
\begin{example} \label{EX: r=3} The simplest cases of (i) are $1^2 + 2^2 \equiv 3^2\!\pmod{2^2}$ and
\begin{equation*}
	1^{12} + 2^{12} + \dotsb + 42^{12} \equiv 43^{12} \pmod{42^2}.
\end{equation*}
An example of (ii) is $1^4 + 2^4 \equiv 3^4\!\pmod{2^3}$. (More generally, one can show that
\begin{equation*}
	1^n + 2^n \equiv 3^n \pmod{2^d},\quad \text{if } 2^{d-1}\mid n,
\end{equation*}
for any positive integers $n$ and $d$.)
\end{example}

In light of Theorem~\ref{THM: EMCongT} and Corollary~\ref{COR: Sigma_n(k) mod p^2}, it is natural to ask whether Theorem~\ref{THM: p^2} has an analogous corollary about supercongruences modulo $p^3$.

\begin{cnj}  \label{CNJ: p^3}
If $1^n + 2^n + \dotsb + k^n\equiv (k+1)^n\!\pmod{k^2}$ and prime $p \mid k$, then
\begin{equation*}
	\Sigma_n(k)\equiv \frac{k}{p}\,\Sigma_n(p) \pmod{p^3}.
\end{equation*}
\end{cnj}

\begin{example}
For $p=2,3,7$, one can compute that
\begin{equation*}
	1^{12} + 2^{12} + \dotsb + 42^{12}
		\equiv \frac{42}{p}\,\left(1^{12} + 2^{12} + \dotsb + p^{12} \right)\!\pmod{p^3}.
\end{equation*}
\end{example}

In fact, for $p=2,3,7$ it appears that $\Sigma_n(42)\equiv \dfrac{42}{p}\,\Sigma_n(p)\!\pmod{p^3}$ holds true not only when $n\equiv 12\!\pmod{42}$, but indeed for all $n\equiv 0\!\pmod{6}$. One reason may be that, for $p=7$ (but not for $p=2$ or $3$), apparently $6\mid n$ implies $p^2\mid \Sigma_{n-1}(p)$. (Compare $p\mid \Sigma_{n-1}(p)$ in the proof of Corollary~\ref{COR: Sigma_n(k) mod p^2}.)

%%%  ACKNOWLEDGMENTS
\section*{Acknowledgments}
We are very grateful to Wadim Zudilin for contributing Theorem~\ref{THM: p^2} and some of the other results in Section~\ref{SEC: Super}. The first author thanks both the Max Planck Institute for Mathematics for its hospitality during his visit in October 2008 when part of this work was done, and Pieter Moree for reprints and discussions of his articles on the \EME{}. The second author thanks Angus MacMillan and Dr.\ Stanley K.\ Johannesen for supplying copies of hard-to-locate papers, and Drs.\ Jurij and Daria Darewych for underwriting part of the research.

%%%  BIBLIOGRAPHY

\end{document}